\documentclass[a4paper,11pt,fleqn]{amsart}

\usepackage[utf8]{inputenc}
\usepackage[T1]{fontenc}

\usepackage[symbol,perpage]{footmisc}

\usepackage{enumitem}

\usepackage[foot]{amsaddr}

\title{Trace theory for Sobolev mappings into a manifold}

\author{Petru Mironescu}

\address[P. Mironescu]{
Universit\'e de Lyon\\
Universit\'e Lyon 1\\
CNRS UMR 5208 Institut Camille Jordan\\
43 Boulevard du 11 Novembre 1918\\
69622 Villeurbanne Cedex\\
France}

\address[P. Mironescu]{Simion Stoilow Institute of Mathematics of the Romanian Academy\\
 Calea Grivi\c tei 21\\
 010702 Bucure\c sti\\
 Rom\^ania
}
\thanks{This work has been initiated  during a long term visit  of P. Mironescu at the  Simion Stoilow Institute of Mathematics of the Romanian Academy; he thanks the Institute and the Centre Francophone en Math\'ematiques in Bucharest for their support on that occasion.}

\author{Jean Van Schaftingen}
\address[J. Van Schaftingen]{Universit\'e catholique de Louvain\\
Institut de Recherche en Math\'ematique et Physique\\
Chemin du Cyclotron 2 bte L7.01.01\\
1348 Louvain-la-Neuve\\
Belgium}
\email{mironescu@math.univ-lyon1.fr, Jean.VanSchaftingen@uclouvain.be}

\thanks{J. Van Schaftingen was supported by the Mandat d'Impulsion Scientifique F.4523.17, \enquote{Topological singularities of Sobolev maps} of the Fonds de la Recherche Scientifique--FNRS}

\usepackage{geometry}
\usepackage{microtype}
\usepackage{lmodern}

\usepackage{amssymb}

\usepackage{mathtools}

\usepackage{esint}
\usepackage{xfrac}
\usepackage{csquotes}

\usepackage[foot]{amsaddr}
\usepackage{constants}

\usepackage{hyperref}

\usepackage[abbrev,backrefs]{amsrefs}

\makeatletter
 \hypersetup{pdftitle={Trace theory for Sobolev mappings into a manifold},%
   pdfauthor={Petru Mironescu and Jean Van Schaftingen}}
 \makeatother

\usepackage[nameinlink%
,capitalize%
]{cleveref}

\newtheorem{theorem}{Theorem}
\newtheorem{proposition}{Proposition}[section]

\newtheorem{corollary}[proposition]{Corollary}

\theoremstyle{definition}

\newtheorem{openproblem}{Open problem}

\theoremstyle{remark}

\numberwithin{equation}{section}

\DeclarePairedDelimiter{\parens}{\lparen}{\rparen}
\DeclarePairedDelimiter{\abs}{\lvert}{\rvert}
\DeclarePairedDelimiter{\set}{\lbrace}{\rbrace}
\DeclarePairedDelimiter{\norm}{\lVert}{\rVert}

\DeclarePairedDelimiter{\floor}{\lfloor}{\rfloor}

\newcommand{\familyname}[1]{\textsc{#1}}

\newcommand{\dif}{\,\mathrm{d}}

\newcommand{\st}{;\, }

\newcommand{\manifold}[1]{\mathcal{#1}}
\newcommand{\Rset}{\mathbb{R}}
\newcommand{\Nset}{\mathbb{N}}
\newcommand{\Sset}{\mathbb{S}}
\newcommand{\Bset}{\mathbb{B}}
\newcommand{\Zset}{\mathbb{Z}}
\newcommand{\compose}{\,\circ\,}
\newcommand{\restr}[1]{\!\upharpoonright_{#1}}
\newcommand{\defeq}{\triangleq}
\newcommand{\Deriv}{\mathrm{D}}

\DeclareMathOperator{\trace}{tr}

\makeatletter
\newcommand\footnoteref[1]{\protected@xdef\@thefnmark{\ref{#1}}\@footnotemark}
\makeatother

\DefineFNsymbols{symbols}{{\ensuremath{\dagger}}{\ensuremath{\ddagger}}\S\P
{\ensuremath{\|}}{\ensuremath{\dagger\dagger}}{\ensuremath{\ddagger\ddagger}}}
\setfnsymbol{symbols}

\usepackage[english]{babel}
\begin{document}

\begin{abstract}
We review the current state of the art concerning the characterization of traces of the  spaces $W^{1, p} (\mathbb{B}^{m-1}\times (0,1), \mathcal{N})$ of Sobolev mappings with values  into a compact manifold $\mathcal{N}$.
In particular, we exhibit a new analytical obstruction to the extension, which occurs when $p < m$ is an integer and the homotopy group \(\pi_p (\mathcal{N})\) is non trivial. On the positive side, we prove the surjectivity of the trace operator when the fundamental group $\pi_1 (\mathcal{N})$ is finite and $\pi_2 (\mathcal{N}) \simeq \dotsb \simeq \pi_{\lfloor p - 1 \rfloor} (\mathcal{N}) \simeq \{0\}$.
We present several open problems connected to the extension problem.
\end{abstract}


\maketitle

\section{Introduction}

The classical trace theory characterizes the boundary values of functions in the linear Sobolev spaces
\(W^{1, p} (\Rset^{m - 1} \times (0, 1), \Rset)\), with \( m\ge 2\) and \(1 \le p < \infty\). These spaces  are  defined as
\begin{equation*}
    W^{1, p}
      (\Rset^{m - 1} \times (0, 1), \Rset)
  \defeq
  \set[\big]{
U: \Rset^{m - 1} \times (0, 1)\to \Rset
\st
U\in L^p \text{ and } DU\in L^p
}.
\end{equation*}

The characterization of the traces involves the \emph{fractional Sobolev--Slobodecki{\u \i} space} \(W^{1-\sfrac{1}{p},p}(\Rset^\ell, \Rset)\). Recall that, when \(0<s<1\), the fractional spaces \(W^{s,p}(\Rset^\ell, \Rset)\) are defined as
\begin{equation*}
W^{s, p}
(\Rset^\ell, \Rset)
\defeq
\set[]{
  u : \Rset^{\ell} \to \Rset
  \st u\in L^p \text{ and }\mathcal{E}^{s, p} (u) < \infty
},
\end{equation*}
where the \emph{fractional Gagliardo energy} \(\mathcal{E}^{s, p} (u)\) of a measurable function \(u :  \Rset^{\ell} \to \Rset\) is given by
\begin{equation*}
\mathcal{E}^{s, p} (u)
\defeq
\iint\limits_{\Rset^{\ell} \times \Rset^{\ell}}
\frac{\abs{u (y) - u (x)}^p}{\abs{y - x}^{\ell + sp}}\dif y \dif x.
\end{equation*}
The fractional spaces \(W^{s, p} (\Rset^\ell, \Rset)\) can also be characterized as interpolated spaces of \(L^p(\Rset^\ell, \Rset)\) and \(W^{1,p}(\Rset^\ell, \Rset)\) \cite{Lions_Peetre_1964}*{Théorème VI.2.1} (see also \cite{Adams_Fournier_2003}*{Theorem 7.39}).

The central result in classical trace theory, due to E.  \familyname{Gagliardo} \cite{Gagliardo_1957} (see also \citelist{\cite{diBenedetto_2016}*{\S 10.17--10.18 and Proposition 17.1}\cite{Mazya_2011}*{\S 10.1.1}}), asserts that, when \(p > 1\),  there exists a unique
linear continuous surjective trace operator
\(\trace_{\Rset^{m - 1} \times \{0\}} : W^{1, p} (\Rset^{m - 1} \times (0, 1), \Rset) \to W^{1 - \sfrac{1}{p}, p} (\Rset^{m - 1}, \Rset)\), extending the (pointwise) trace on \(\Rset^{m - 1} \times \{0\}\) of smooth maps \(U\in C^\infty(\Rset^{m-1}\times [0,1), \Rset)\cap W^{1,p}(\Rset^{m-1}\times (0,1), \Rset)\).
Moreover, the operator \(\trace_{\Rset^{m - 1} \times \{0\}}\) has a linear continuous right inverse. The harmonic extension (convolution with the Poisson kernel), the heat semigroup (convolution with the heat kernel)
or, more generally,  the  convolution with appropriate families of mollifiers are explicit examples of such  right inverses. For example, if \(u\in W^{1-\sfrac{1}{p},p}(\Rset^{m-1}, \Rset)\), then its harmonic extension \(U\) to \(\Rset^{m-1}\times (0,\infty)\), restricted to \(\Rset^{m-1}\times (0,\infty)\), is an extension of \(u\) in the sense that it belongs to \(W^{1,p}(\Rset^{m-1}\times (0,1), \Rset)\) and has trace \(u\) on \(\Rset\times \{0\}\).

When \(p = 1\), the trace operator is a \emph{linear continuous surjection} on \(L^1 (\Rset^{m  - 1})\) \cite{Gagliardo_1957} that has \emph{no linear continuous right inverse} (J. \familyname{Peetre} \cite{Peetre_1979}).

Trace theory has local versions, in which the whole Euclidean space \(\Rset^{m-1}\) is replaced by a Lipschitz domain. For simplicity, we focus on the case of the unit ball \(\Bset^{m-1}\). With \(1\le p<\infty\) and  \(0<s<1\), the adapted Sobolev spaces and fractional energies are
\begin{gather*}
    W^{1, p}
      (\Bset^{m - 1} \times (0, 1), \Rset)
  \defeq
  \set[\big]{
U: \Bset^{m - 1} \times (0, 1)\to \Rset
\st
U\in L^1_{\mathrm{loc}} \text{ and } DU\in L^p
},
\\
W^{s, p}
(\Bset^\ell, \Rset)
\defeq
\set[]{
  u : \Rset^{\ell} \to \Rset
  \st u\text{ is measurable and }\mathcal{E}^{s, p} (u) < \infty
},
\intertext{and}
\mathcal{E}^{s, p} (u)
\defeq
\iint\limits_{\Bset^{\ell} \times \Bset^{\ell}}
\frac{\abs{u (y) - u (x)}^p}{\abs{y - x}^{\ell + sp}}\dif y \dif x.
\end{gather*}

In this framework, when \(p>1\), the trace operator
\begin{equation*}
\trace_{\Bset^{m - 1} \times \{0\}} : W^{1, p} (\Bset^{m - 1} \times (0, 1), \Rset) \to W^{1 - \sfrac{1}{p}, p} (\Bset^{m - 1}, \Rset)
\end{equation*}
is a  linear continuous surjection that has a linear continuous right inverse. Again, explicit extensions of maps in \(W^{1 - \sfrac{1}{p}, p} (\Bset^{m - 1}, \Rset)\) can be  obtained via convolutions with appropriate mollifiers.

\medbreak
The previous considerations extend readily to the case where the target space  \(\Rset\) is replaced  by a finite-dimensional Euclidean space \(\Rset^\nu\), where \(\nu \in \Nset^*\).

\medbreak

When \(\manifold{N} \subset \Rset^\nu\) is an embedded compact Riemannian submanifold\footnote{By Nash's embedding theorem \cite{Nash_1956}, such an embedding exists for any abstract Riemannian manifold.}, we consider the corresponding Sobolev spaces of mappings into the manifold \(\manifold{N}\), defined, for \(1\le p<\infty\) and \(0<s<1\), as
\begin{gather*}
  \begin{split}
W^{1, p} (\Bset^{m - 1} \times (0, 1), \manifold{N})
\defeq
\set[]
{
  U \in W^{1, p} (&\Bset^{m - 1} \times (0, 1), \Rset^\nu)
  \st
  U \in \manifold{N}\\
  & \text{ almost everywhere in }
  \Bset^{m - 1} \times (0, 1)
}
\end{split}
\intertext{and}
W^{s, p} (\Bset^{\ell}, \manifold{N})
\defeq
\set[]
{
  u \in W^{s, p} (\Bset^{\ell})
  \st
  u \in \manifold{N}
  \text{ almost everywhere in }
  \Bset^{\ell}
}.
\end{gather*}

The classical linear theory readily implies  that
\begin{equation}
  \label{eq_saiSh5rei7chaecie}
  \trace_{\Bset^{m - 1} \times \{0\}}  \parens[\big]{W^{1, p} (\Bset^{m - 1} \times (0, 1), \manifold{N})}
\subseteq W^{1 - \sfrac{1}{p}, p} (\Bset^{m - 1}, \manifold{N})
\end{equation}
(with the convention \(W^{0,1}(\Bset^{m - 1}, \manifold{N})=L^1(\Bset^{m - 1}, \manifold{N})\)).

The basic question of the trace and extension theory for Sobolev mappings with values into manifolds is to determine whether
equality holds in the inclusion \eqref{eq_saiSh5rei7chaecie}; the linear trace theory merely provides an extension taking its values into the ambient Euclidean space  \(\Rset^\nu\) and the problem is to determine whether \emph{every} map \(u\in W^{1-\sfrac{1}{p},p}(\Bset^{m-1} , \manifold{N})\) has a \(W^{1,p}(\Bset^{m-1} \times (0, 1), \Rset^\nu)\) extension  with values into  \(\manifold{N}\).
If this holds, then \(W^{1-\sfrac{1}{p},p}(\Bset^{m-1} , \manifold{N})\) has the \emph{extension property}.

\medbreak
Let us start by noting a harmless condition in order to study the extension property: the manifold \({\manifold N}\) will be connected. Indeed, if \(U\in W^{1,1}_{\mathrm{loc}}(\Bset^{m-1}\times (0,1), \manifold{N})\), then the essential range of the map \(U\) is connected \cite{bmbook}*{Theorem 7.5}, and thus the mapping \(U\) takes values into a connected component of \(\manifold{N}\); therefore, so does its trace.
When \(p < 2\), there exists a map \(u \in W^{1 - \sfrac{1}{p}, p} (\Bset^{m - 1}, \manifold{N})\) that take constant values on smooth subsets of \(\Bset^{m - 1}\), and therefore \emph{we have to assume} that the manifold \(\manifold{N}\) is necessarily connected.
On the other hand when \(p \ge 2\), the essential range of any map \(u \in W^{1 - \sfrac{1}{p}, p} (\Bset^{m - 1}, \manifold{N})\) is connected and there is thus \emph{no loss of generality} to work with a connected target manifold \(\manifold{N}\).
We assume henceforth that the \emph{manifold \(\manifold{N}\) is connected}.

\medbreak
In the case of subcritical dimensions \(m \le p\), the answer to the trace and extension problem is positive.

\begin{theorem}
  \label{theorem_trace_subcritical}
If \(m \le p\), then \(W^{1-\sfrac{1}{p},p}(\Bset^{m-1} , \manifold{N})\) has the extension property.
\end{theorem}

\Cref{theorem_trace_subcritical} is due to F. \familyname{Bethuel} and
F. \familyname{Demengel} \cite{Bethuel_Demengel_1995}*{Theorems 1 \& 2}. Its proof relies on the fact that, when \(m \ge p\), an extension by convolution of a map \(u\in W^{1-\sfrac{1}{p},p}(\Bset^{m-1}, \manifold{N})\) takes, in a neighborhood of \(\Bset^{m - 1} \times \{0\}\),  its values in a small tubular neighborhood of \(\manifold{N}\). This important observation has roots in the seminal work of  R. \familyname{Schoen} and K.  \familyname{Uhlenbeck} \citelist{\cite{Schoen_Uhlenbeck_1982}*{\S 3}\cite{Schoen_Uhlenbeck_1983}*{\S 4}} on \(H^1\) maps with values into manifolds; see also H. \familyname{Brezis} and L. \familyname{Nirenberg} \cite{Brezis_Nirenberg_1995}  for far-reaching consequences of properties of this type in connection with the degree theory for VMO maps with values into manifolds.

\medbreak
In higher dimensions \(m > p\),  the answer to the trace problem is also positive provided the integrability exponent \(p\) is small.

\begin{theorem}
  \label{theorem_trace_sufficient}
  If \(1 \le p < 2\), then \(W^{1-\sfrac{1}{p},p}(\Bset^{m-1} , \manifold{N})\) has the extension property.
 \end{theorem}

 \cref{theorem_trace_sufficient} is due to is due to R. \familyname{Hardt} and F.H. \familyname{Lin} \cite{Hardt_Lin_1987}*{Theorem 6.2}.\footnote{Strictly speaking, the case \(p=1\), which is an exceptional case for trace theory, is not specifically considered in \cite{Hardt_Lin_1987}. However, \cref{theorem_trace_sufficient} with \(p=1\) and \cref{theorem_extension_estimates} \ref{it_quaes1viegheemahvieJohZ2} are proved exactly as the corresponding results for \(1<p<2\). The initial ingredient is the existence, for each measurable map \(u:\Bset^{m-1}\to{\manifold N}\), of some extension \(U\in W^{1,1}(\Bset^{m-1}\times (0,1), \Rset^\nu)\) such that
  \[
  \norm{\nabla U}_{L^1}\le C \norm[\bigg]{u-\fint_{\Bset^{m-1}}u}_{L^1}.
  \]}

In particular, when \(m = 2\), the whole range of integrability exponents \(1 \le p < \infty\) is covered by the combination of \cref{theorem_trace_subcritical} and \cref{theorem_trace_sufficient}.
A hint  to the absence of any topological condition beyond connectedness of the manifold \(\manifold{N}\) is the fact that, when \(0 < sp < 1\), the space \(W^{s, p} (\Bset^{m - 1}, \Rset^\nu)\)
contains characteristic functions of smooth sets and  hence topological obstructions cannot arise in these spaces. (A similar phenomenon arises for the lifting problem when \(0 < sp < 1\) \citelist{\cite{Bourgain_Brezis_Mironescu_2000}\cite{Bethuel_Chiron_2007}}.)

\medbreak
However, when \(2 \le p < m\), one encounters some obstructions in the extension problem.
A first  example is provided by the \emph{topological obstruction}.

\begin{theorem}
  \label{theorem_topological_obtsruction}
If \(2 \le p < m \) and if \(\pi_{\floor{p - 1}} (\manifold{N}) \not \simeq \{0\}\)\footnote{Here and in what follows, \(\floor{t} \in \Zset\) denotes the integer part of the real number \(t \in \Rset\).},
then \(W^{1 - \sfrac{1}{p}, p} (\Bset^{m - 1}, \manifold{N})\) does not have the extension property.
\end{theorem}

\Cref{theorem_topological_obtsruction} is due to R. \familyname{Hardt} and F.H. \familyname{Lin}  \cite{Hardt_Lin_1987}*{\S 6.3} and F. \familyname{Demengel} and F.  \familyname{Bethuel} \cite{Bethuel_Demengel_1995}*{Theorem 4}.
An equivalent formulation of the above topological obstruction is the following:  there exists a map \(f \in C^1 (\Sset^{\floor{p - 1}}, \manifold{N})\) that cannot be extended continuously to the ball \(\Bset^{\floor{p}}\). Given such an  \(f\), an explicit example of a map \(u\in W^{1 - \sfrac{1}{p}, p} (\Bset^{m - 1}, \manifold{N})\) with no extension \(U\in  W^{1, p} (\Bset^{m - 1}\times (0,1), \manifold{N})\) is given by
\begin{equation}
\label{A1}
u (x', x'') \defeq f (x'/\abs{x'}),\ \forall\,  (x', x'') \in \Rset^{\floor{p}} \times \Rset^{m - 1 - \floor{p}}\text{ such that }(x', x'')\in \Bset^{m - 1}.
\end{equation}

By the above, in the range \(2\le p<m\), a necessary condition for the extension property to hold  is \(\pi_{\floor{p - 1}} (\manifold{N}) \simeq \{0\}\).
 When \(2 \le p < 3\le m\), this condition becomes  \(\pi_1(\manifold{N}) \simeq \{0\}\), i.e,  \(\manifold{N}\) has to be simply connected.
 It turns out that this condition is also sufficient.

\begin{theorem}
  \label{theorem_trace_sufficient_small_p}
  If \(2 \le p < 3 \le m\) and if \(\pi_1 (\manifold{N}) \simeq \{0\}\),
  then  \(W^{1 - \sfrac{1}{p}, p} (\Bset^{m - 1}, \manifold{N})\) has the extension property.
\end{theorem}

\Cref{theorem_trace_sufficient_small_p} is due to R. \familyname{Hardt} and  F.H. \familyname{Lin} \cite{Hardt_Lin_1987}*{Theorem 6.2}.

\medbreak
Besides the topological obstruction, the extension problem encounters some
\emph{analytical obstructions}.

\begin{theorem}
  \label{theorem_analytical_obstruction}
Assume \(2 \le p < m \). If
\begin{enumerate}[label=\alph*)]
  \item \label{it_uX7Veuch7fai0ri3eive7ki3}
either \(\pi_{\ell} (\manifold{N})\) is infinite for some
  \(\ell \in \{1, \dotsc, \floor{p  - 1}\}\)
  \item \label{it_heeShiechee0theek7iL1Ees}
  or \(p \in \Nset\) and \(\pi_{p - 1} (\manifold{N}) \not \simeq\{0\}\),
\end{enumerate}
then there exists some smooth map \(u\in W^{1-\sfrac{1}{p},p}(\Bset^{m-1}, \manifold{N})\) that has no extension \(U\in W^{1,p}(\Bset^{m-1}\times (0,1), \manifold{N})\). \\
In particular, \(W^{1 - \sfrac{1}{p}, p} (\Bset^{m - 1}, \manifold{N})\) does not have the extension property.
\end{theorem}

\medbreak
\Cref{theorem_analytical_obstruction} \ref{it_uX7Veuch7fai0ri3eive7ki3} is due to F. \familyname{Demengel} and F. \familyname{Bethuel} when \(\ell = 1\) \cite{Bethuel_Demengel_1995}*{Theorem 4} and
to F. \familyname{Bethuel} for a general \(\ell\) \cite{Bethuel_2014}\footnote{The triviality of the groups \(\pi_1 (\manifold{N}), \dotsc, \pi_{\ell - 1} (\manifold{N})\) and the non-triviality of \(\pi_{\ell} (\manifold{N})\) (which are the only explicit assumptions in \cite{Bethuel_2014}) do not imply that \(\pi_{\ell} (\manifold{N})\) is  infinite; see \cref{proposition_manifold_EMcL}. However,  the latter property is used in the construction of maps with arbitrary large topological energy \cite{Bethuel_2014}*{Lemma 2.2}. }.  \Cref{theorem_analytical_obstruction} \ref{it_heeShiechee0theek7iL1Ees} is one of the contributions of the present  work  (see \cref{section_obstr_nonest} below).

The map \(u\) given in \cref{theorem_analytical_obstruction} is not smooth up to the boundary. However, it is  the strong limit of maps smooth up to the boundary, obtained from \(u\) by suitable dilations of the domain.
Note the difference in nature with the counterexample in \eqref{A1}; there, \(u\) has  strong interior singularities in the set  \(\Bset^{m-1}\cap (\{0\}\times\Rset^{m - 1 - \floor{p}})\).

When \(p\) is an integer, \cref{theorem_analytical_obstruction} \ref{it_heeShiechee0theek7iL1Ees} implies that the assumption that \(\pi_{p - 1} (\manifold{N})\) is trivial plays, in the extension problem,  a role even for the strong limits of smooth maps (and is not only required just to have the strong density of smooth maps \citelist{\cite{Brezis_Mironescu_2015}\cite{Bousquet_Ponce_Van_Schaftingen_2014}}).

\medbreak
On the positive side, we have the following result.

\begin{theorem}
  \label{theorem_trace_sufficient_medium_p}
  If \(3 \le p < m\), if \(\pi_1 (\manifold{N})\) is finite and if \(\pi_2 (\manifold{N}) \simeq \dotsb \simeq \pi_{\floor{p - 1}} \simeq \{0\}\),
  then
  \(W^{1 - \sfrac{1}{p}, p} (\Bset^{m - 1}, \manifold{N})\) has the extension property.
\end{theorem}

\Cref{theorem_trace_sufficient_medium_p} is due to R. \familyname{Hardt} and F.H. \familyname{Lin}  \cite{Hardt_Lin_1987}*{Theorem 6.2} when \(\pi_1 (\manifold{N})\) is trivial. In full generality,  it is proved in the present work  (see \cref{section_extension} below). The proof strongly relies on  an idea of  F. \familyname{Bethuel} \cite{Bethuel_2014}*{Theorem 1.5 (iii)} and uses a very  recent result on the lifting over compact covering spaces \cite{Mironescu_VanSchaftingen}.

\medbreak
Combining \cref{theorem_trace_sufficient,theorem_topological_obtsruction,theorem_trace_sufficient_small_p,theorem_analytical_obstruction,theorem_trace_sufficient_medium_p},
we obtain the following.

\begin{corollary}
\label{A2}
Assume \(m\ge p\).
\begin{enumerate}[label=\arabic*.]
\item
If \(1\le p<2\), then \(W^{1 - \sfrac{1}{p}, p} (\Bset^{m - 1}, \manifold{N})\) has the extension property.
\item
If \(2\le p<3\), then \(W^{1 - \sfrac{1}{p}, p} (\Bset^{m - 1}, \manifold{N})\) has the extension property if and only if \(\pi_1({\manifold N})\simeq\{ 0\}\).
\item
If \(3\le p<4\), then \(W^{1 - \sfrac{1}{p}, p} (\Bset^{m - 1}, \manifold{N})\) has the extension property if and only if \(\pi_1({\manifold N})\) is finite and \(\pi_2 (\manifold{N})\simeq\{ 0\}\).
\end{enumerate}
\end{corollary}

What happens when  \(4 \le p < m\) (assuming the necessary conditions for the extension property imposed by \cref{theorem_topological_obtsruction} and \cref{theorem_analytical_obstruction}) is  \emph{terra incognita}.

\begin{openproblem}
\label{A3}
Assume \(4 \le p < m\),  \(\pi_1 (\manifold{N}), \dotsc, \pi_{\floor{p - 2}} (\manifold{N})\)  finite and
  \(\pi_{\floor{p - 1}} (\manifold{N})\)  trivial. Does \(W^{1 - \sfrac{1}{p}, p} (\Bset^{m - 1}, \manifold{N})\) have the extension property?  \end{openproblem}

F. \familyname{Bethuel} and F.  \familyname{Demengel} have conjectured that the answer to \cref{A3}  is  positive \cite{Bethuel_Demengel_1995}*{Conjecture 2}. Let us note that there exist manifolds satisfying the assumptions of \cref{A3}  (see \cref{proposition_manifold_EMcL}).

\medbreak
We next turn to the \emph{quantitative form} of the extension problem, more specifically the existence of \(U\) whose energy is controlled in terms of the one of \(u\). Given \(u\in W^{1-\sfrac{1}{p},p}(\Bset^{m-1}, \manifold{N})\), a natural \emph{extension energy} is
\begin{equation}
  \label{eq_cheem1shieThair3m}
\mathcal{E}^{1, p}_\mathrm{ext} (u)
\defeq
\inf\, \set[\bigg]
{
  \int_{\Bset^{m - 1} \times (0, 1)} \abs{\Deriv U}^p
  \st U   \text{ is an extension of }  u
}.
\end{equation}

The next result shows that, under the topological assumptions in \cref{theorem_trace_sufficient},    \cref{theorem_trace_sufficient_small_p} or \cref{theorem_trace_sufficient_medium_p}, the extension energy is controlled linearly.

\begin{theorem}
  \label{theorem_extension_estimates}
  \begin{enumerate}[label=\arabic*.]
\item
If
\begin{enumerate}[label=\alph*)]
\item[a)] either \(1 < p < 2\),
\item[b)] or \(2 \le p < 3\) and \(\pi_1 (\manifold{N})\simeq \{0\}\),
\item[c)] or \(3 \le p < \infty\), \(\pi_1 (\manifold{N})\) is finite and \(\pi_2 (\manifold{N}) \simeq \dotsb \simeq \pi_{\floor{p - 1}} (\manifold{N}) \simeq \{0\}\),
\end{enumerate}
then there exists a constant \(C=C(p, m, \manifold{N})\) such that, for every  mapping \(u \in W^{1 - \sfrac{1}{p}, p} (\Bset^{m - 1}, \manifold{N})\),
\begin{equation*}
\mathcal{E}^{1, p}_{\mathrm{ext}} (u)
\le
C
\mathcal{E}^{1 - \sfrac{1}{p}, p} (u).
\end{equation*}
\item \label{it_quaes1viegheemahvieJohZ2}
If \(p=1\), then
\begin{equation*}
\mathcal{E}^{1, 1}_{\mathrm{ext}} (u)
\le C \left\|u-\fint_{\Bset^{m-1}}u  \right\|_{L^1}.
\end{equation*}
\end{enumerate}

\end{theorem}

\Cref{theorem_extension_estimates} is a direct consequence of the estimates resulting from the proofs of \cref{theorem_trace_sufficient,theorem_trace_sufficient_small_p,theorem_trace_sufficient_medium_p}. Note that we do not require \(p<m\). In the range
\(1<p < m\), \Cref{theorem_extension_estimates} follows  without any calculation from the existence results in the above theorems and  an abstract nonlinear uniform boundedness principle due to A. \familyname{Monteil} and J.  \familyname{Van Schaftingen} \cite{Monteil_Van_Schaftingen_2019}*{Theorem 1.1}.

\begin{theorem}
  \label{theorem_no_linear_bounded}
  %TODO Check with dimensions
  Assume that \(\ell \in \{1, \dotsc, m - 1\}\). Let \(b\in {\manifold N}\). If
  \begin{enumerate}
    \item[a)] either \(\ell < p - 1\) and \(\pi_\ell (\manifold{N})\) is infinite,
    \item[b)] or \(\ell = p - 1\) and \(\pi_\ell (\manifold{N})\) is nontrivial,
  \end{enumerate}
 then there exists a sequence \((u_j)_{j \in \Nset}\) in \(C^\infty_b (\overline\Bset^{m - 1}, \manifold{N})\)\footnote{\label{c8}Here and in what follows, the subscript \(b\) denotes classes of maps with trace \(b\) on the boundary.}  such that
  \begin{align}
  \label{e5}
    \liminf_{j \to \infty} \mathcal{E}^{1 - \sfrac{1}{p}, p} (u_j) &> 0 &
   &\text{ and }&
    \lim_{j \to \infty} \frac{\mathcal{E}^{1, p}_{\mathrm{ext}} (u_j)}{\mathcal{E}^{1 - \sfrac{1}{p}, p} (u_j)}
     &= \infty.
  \end{align}
\end{theorem}

\begin{openproblem}
If \(p \ge 4\), and \(\pi_1 (\manifold{N}), \dotsc, \pi_{\floor{p - 2}} (\manifold{N})\) are finite and if
\(\pi_{\floor{p - 1}} (\manifold{N})\) is trivial, does there exist a constant \(C\) such that
for every \(u \in W^{1 - \sfrac{1}{p}, p} (\Bset^{m - 1}, \manifold{N})\), one has
\begin{equation*}
\mathcal{E}^{1, p}_{\mathrm{ext}} (u)
\le
C
\mathcal{E}^{1 - \sfrac{1}{p}, p} (u)?
\end{equation*}
\end{openproblem}

\medskip

In the cases where the trace operator is not surjective, a natural question is to describe the elements in the trace space, in a similar fashion to what has been done in many cases for the strong approximation by smooth maps of Sobolev mappings \citelist{\cite{Bethuel_Coron_Demengel_Helein_1991}\cite{Bethuel_1990}\cite{Riviere_2000}}.

\begin{openproblem}
  Characterize the trace space \(\trace_{\Bset^{m - 1} \times \{0\}}  (W^{1, p} (\Bset^{m - 1}\times (0, 1), \manifold{N}))\).
\end{openproblem}

When either \(1 \le p < 2\) or \(p \ge m\), then by \cref{theorem_trace_subcritical,theorem_trace_sufficient} the trace space is the fractional Sobolev space \(W^{1 - \sfrac{1}{p}, p}(\Bset^{m - 1}, \manifold{N})\).
When \(2 \le p < 3\), a map \(u\)  is in the trace space if and only \(u\) has a \(W^{1-\sfrac 1p,p}\) lifting in the universal covering of \({\manifold N}\).
\footnote{More precisely, let \({\widetilde{\manifold N}}\)  be the universal covering of \({\manifold N}\) and
 \(\pi:\widetilde{\manifold N}\to {\manifold N}\)
 the corresponding covering map.
 Then \(u\in \trace_{\Bset^{m - 1} \times \{0\}}  (W^{1, p} (\Bset^{m - 1}\times (0, 1), {\manifold N}))\) if and only if there exists some \(\varphi\in W^{1-\sfrac 1p, p}(\Bset^{m - 1}, \widetilde{\manifold N})\) such that \(u=\pi\compose\varphi\).
}  (This assertion can be established  by adapting the proof of Theorem 3 in \cite{Bethuel_2014}; see also Section \ref{section_extension}.)  However, currently there is no tractable characterization of the mappings having this property.

\medbreak
A partial result in this direction has been obtained by   B.  \familyname{White} \cite{White_1988}*{Theorem 4.1}, who characterized  maps in  \(\trace_{\Bset^{m-1}\times \{0\}} (W^{1,p}(\Bset^{m-1}\times (0,1), {\manifold N}))\) that are in addition Lipschitz-continuous.

\medbreak
When \(p \in \Nset\), the trace spaces can be characterized by a topological condition on generic skeletons and the boundedness of families of Ginzburg--Landau energies remaining bounded when the order parameter goes to \(0\) \citelist{\cite{Bourgain_Brezis_Mironescu_2004}\cite{Isobe_2003}}; it would be desirable to have a more intrinsic criterion, probably relying on the behaviour of the map on \(\floor{p - 2}\)--dimensional skeletons. In view of the quantitative obstructions to the extension problem \cite{Bethuel_2014}, the condition should be quantitative, in contrast with the more qualitative criteria for the strong approximation by smooth maps.

\medbreak
Up to this point, we have considered the problem of traces on \(\Bset^{m - 1} \times \{0\}\)
of maps on \(\Bset^{m - 1} \times (0, 1)\). More generally, one can consider a manifold \(\manifold{M}\) with a boundary \(\partial \manifold{M}\) and traces on \(\manifold{M}\).

\begin{openproblem}
  When do we have \(\trace_{\partial \manifold{M}}  (W^{1,p} (\manifold{M},\manifold{N})) = W^{1 - \sfrac{1}{p}, p} (\partial \manifold{M}, \manifold{N})\)?
\end{openproblem}

R. \familyname{Hardt}  and F.H. \familyname{Lin}  have proved that this is the case when \(\pi_1 (\manifold{N}) \simeq \dotsb \simeq \pi_{\floor{p - 1}} (\manifold{N}) \simeq \{0\}\) \cite{Hardt_Lin_1987}*{Theorem 6.2}.
On the other hand,
if \(\trace_{\partial \manifold{M}}  (W^{1,p} (\manifold{M},\manifold{N})) = W^{1 - \sfrac{1}{p}, p} (\partial \manifold{M}, \manifold{N})\), then \({\manifold M}\) and \({\manifold N}\) have to satisfy the following topological property: For some arbitrary\footnote{A homotopy equivalence argument shows that the condition does not depend on the triangulation.} triangulation \({\manifold T}\) of the manifold \({\manifold M}\), every \({\manifold N}\)-valued  continuous map on the \(\floor{p - 1}\)--skeleton of \({\manifold T} \cap \partial \manifold{M}\) admits a continuous \({\manifold N}\)-valued extension to the \(\floor{p }\)--skeleton of \({\manifold T}\) (T. \familyname{Isobe} \cite{Isobe_2003}; see also \cite{Bethuel_Demengel_1995}*{Theorem 5}).

\medbreak
The linear trace theory extends to weighted spaces \cite{Uspenskii_1961} (see also \cite{Mironescu_Russ_2015}):
if one sets
\begin{multline*}
W^{1, p}_\gamma (\Bset^{m - 1} \times (0, 1))
\\
=
\set[\bigg]{
  U \in W^{1, 1}_{\mathrm{loc}} (\Bset^{m - 1} \times (0, 1))
  \st \iint_{\Bset^{m - 1}\times (0, 1)} \abs{\Deriv U (x, t)}^p \, t^\gamma \dif t \dif x <\infty
  },
\end{multline*}
then, for \(0<s<1\) and \(1\le p<\infty\), we have
\[
\trace_{\Bset^{m - 1} \times \{0\}}  W^{1, p}_{(1 - s)p - 1} (\Bset^{m - 1} \times (0, 1)) = W^{s, p} (\Bset^{m - 1}).
\]

\begin{openproblem} Assume \(0<s<1\) and \(1\le p<\infty\).
Characterize the manifolds for which one has
\[
\trace_{\Bset^{m - 1} \times \{0\}}  W^{1, p}_{(1 - s)p - 1} (\Bset^{m - 1} \times (0, 1)) = W^{s, p} (\Bset^{m - 1}, \manifold{N}).
\]
\end{openproblem}

Finally, if one considers higher-order Sobolev spaces, the derivatives also have traces. It is known for instance that, for \(1\le p<\infty\), we have
\begin{multline*}
\{(\trace_{\Bset^{m - 1} \times \{0\}}  U, \trace_{\Bset^{m - 1} \times \{0\}}  \partial_m U) \st U \in W^{2, p} (\Bset^{m - 1} \times (0, 1))\}\\
= W^{2 - \sfrac{1}{p}, p} (\Bset^m) \times W^{1 - \sfrac{1}{p}, p} (\Bset^m).
\end{multline*}

\begin{openproblem}
  Characterize the manifolds \(\manifold{N}\) such that
  \begin{multline*}
  \{(\trace_{\Bset^{m - 1} \times \{0\}}  U, \trace_{\Bset^{m - 1} \times \{0\}}  \partial_m U) \st U \in W^{2, p} (\Bset^{m - 1} \times (0, 1), \manifold{N})\}\\
  = W^{2 - \sfrac{1}{p}, p} (\Bset^m, \manifold{N}) \times W^{1 - \sfrac{1}{p}, p} (\Bset^m, \manifold{N}).
\end{multline*}
\end{openproblem}

\section{Obstructions and non-estimates}
\label{section_obstr_nonest}
We first prove \cref{theorem_no_linear_bounded} about the obstruction to linear bounds on the extension energy \(\mathcal{E}^{1, p}_{\mathrm{ext}}\).

\begin{proof}[Towards the proof of \cref{theorem_no_linear_bounded}] \emph{A fundamental lower bound when \(p>\ell+1\).} We explain the main idea in \cite{Bethuel_2014}, that we adapt to the context of our topological assumptions. For the convenience of the reader, we first consider  maps defined on  a cylinder, then we adapt the proof to the case of maps defined on half-balls.

Given mappings \(u, v\in C(\overline\Bset^\ell, {\manifold N})\) we consider the following relative homotopy equivalence:  \(v\sim u\) if and only if there exists some \(H\in C(\overline\Bset^\ell\times [0,1], {\manifold N})\) such that \(H(\cdot, 0)=u\), \(H(\cdot, 1)=v\) and \(H(x,t)=u(x)\), \(\forall\, x\in\Sset^{\ell-1}\), \(\forall\, t\in [0,1]\).

We consider some \(U\in W^{1,p}(\Bset^\ell\times (0,1), {\manifold N})\) such that \(\trace_{\Sset^{\ell-1}\times (0,1)} U=b\)\footnote{Recall that \(b\in {\manifold N}\) is a fixed point.}. Identifying  \(U\) with its continuous representative\footnote{\label{gaga}This is possible, by the Morrey embedding, since \(p>\ell+1\).}, we may assume that \(U\in C(\overline\Bset^\ell\times [0,1], {\manifold N})\). Set \(u(x)\defeq U(x,0)\), \(\forall\,  x\in \overline\Bset^\ell\).
By standard trace theory, we have \(U(\cdot, t)=b\) on \(\Sset^{\ell-1}\), \(\forall\,  t\in [0,1]\), and
\begin{equation*}
  \trace_{\Bset^\ell\times\{0\}}U = U \restr{\Bset^\ell\times \{0\}}\simeq u.
\end{equation*}
It follows that
\begin{equation}
\label{b5}
U(\cdot, \tau)\sim u,\ \forall\, s\in [0,1]
\end{equation}
(through the homotopy \([0,1]\ni t\mapsto U(\cdot, (1-t)\tau)\)).

We next deduce a lower bound for the energy of \(U\) as above.
By the Sobolev-Morrey embedding, we have
\begin{equation}
\label{a1}
W^{1,p}(\Bset^\ell, {\manifold N}) \subset W^{1-\sfrac 1p,p}(\Bset^\ell, {\manifold N})\subset C^{0, 1-\sfrac{(\ell+1)} p}(\overline\Bset^\ell, {\manifold N}).
\end{equation}
On the other hand, by standard trace theory the above \(u\) satisfies  \(u\in W^{1-\sfrac 1p,p}(\Bset^\ell, {\manifold N})\). By \eqref{a1}, the  quantity
\begin{equation*}
  \mathcal{E}^{1, p}_\mathrm{top} (u)
  \defeq \inf
 \{
      \mathcal{E}^{1, p} (v)
        \st
        v \in W^{1, p} (\Bset^\ell, \manifold{N}) \text{ and }v\sim u\}
        \end{equation*}
is meaningful.
Combining \eqref{b5} with the fact that \(U(\cdot, s)\in W^{1,p} (\Bset^\ell, \manifold{N})\) for almost every \(\tau \in (0,1)\), we find that
\begin{equation}
\label{a3}
{\mathcal{E}}^{1,p}(U (\cdot, s))\ge  \mathcal{E}^{1, p}_\mathrm{top} (u)\ \text{ for almost every }\tau \in (0,1).
\end{equation}

We next present an analogue of the above on a half-ball.
Set \(\Bset^{\ell+1}_+\defeq \Bset^{\ell+1}\cap \Rset^{\ell+1}_+\). We define similarly \(\Sset^\ell_+\),  and set
\[
 S_+(0, r)=\{ x\in\Rset^{\ell+1}\st \abs{x}=r\text{ and }x_{\ell+1}>0\}.
\]
Let \(S\) be the South pole of \(\Sset^\ell\) and let \(\Psi\) denote the stereographic projection with vertex \(S\) of \(\Sset^{\ell}\setminus\{S\}\) on \(\Rset^\ell\times\{0\}\simeq\Rset^\ell\). Thus \(\Psi\) maps \(\Sset^\ell_+\) onto \(\Bset^\ell\) and leaves invariant \(\Sset^{\ell-1}\times \{0\}\simeq \Sset^{\ell-1}\). Moreover, \(\Psi\) is a bi-Lipschitz homeomorphism from \(\overline\Sset^\ell_+\) to \(\overline\Bset^\ell\).

Let \(U\in W^{1,p}(\Bset^{\ell+1}_+, {\manifold N})\). 
As above, we may assume that \(U\in C(\overline{\Bset^{\ell+1}_+}, \manifold{N})\). 
We set \(u(x)\defeq U(x,0)\), \(\forall\, x\in \overline\Bset^\ell\) and, for \(0<r\le 1\), \(U^r(x)\defeq U(r\, \Psi^{-1}(x))\), \(\forall\, x\in \overline \Bset^\ell_+\), so that \(U^r\in C(\overline\Bset^\ell, {\manifold N})\).
Assume that the map \(u\) has the property that
\begin{equation}
\label{b1} u(x)=b\text{ for }\rho\le \abs{x} \le1.
\end{equation}
We claim that
\begin{equation}
\label{b2}
U^r\sim u,\ \forall\, \rho\le r\le 1.
\end{equation}
Indeed, on the one hand we have \(U^1\sim u\simeq U(\cdot, 0)\) through the homotopy
\begin{equation*}
H(t, x)=U( t\, \Psi^{-1}(x)+t (x,0)),\ \forall\, x\in\overline\Bset^\ell,\, \forall\, t\in [0,1].
\end{equation*}
On the other hand, we have, for \(\rho\le r, r'\le 1\), \(U^r\sim U^{r'}\) through \(H(\cdot ,t)=U^{t r+(1-t)r'}\), \(\forall\, t\in [0,1]\).\footnote{\label{b4}Here, we use \eqref{b1}.}
Combining this with the definition of \(U^r\), we obtain the following analogue of \eqref{a3}:
\begin{equation}
\label{c1}
\int_{S_+(0, r)}|\nabla_T U|^p\ge \Cl{c2}\, r^{\ell-p}\, {\mathcal{E}}^{1,p}_{\mathrm{top}}(u),\ \text{for almost every }\rho<r<1;
\end{equation}
here, \(\Cr{c2}>0\) is an absolute constant, and \(\nabla_T\) stands for the tangential gradient on  the sphere \(S_+(0, r)\).

Integrating the estimate \eqref{c1}, we find that
\begin{equation}
\label{c3}
\int_{\Bset^{\ell+1}_+}\abs{\Deriv U}^p\ge \Cr{c2}\int_\rho^1 r^{\ell-p}\dif r\, {\mathcal{E}}^{1,p}_{\mathrm{top}}(u).
\end{equation}
Taking into account the fact that \(\Bset^{\ell+1}_+\subset \Bset^\ell\times (0,1)\), \eqref{c3} leads to the following fundamental lower bound.
\begin{equation}
{\mathcal{E}}^{1, p}_{\mathrm{ext}}(u)\ge \Cr{c2}{\mathcal{E}}^{1,p}_{\mathrm{top}}(u)
\int_\rho^1 r^{\ell-p} \dif r\, , \ \forall\, u\in W^{1-\sfrac 1p,p}(\Bset^\ell, {\manifold N})\text{ satisfying \eqref{b1}}.
\qedhere\hfill \qed
\label{c4}
\end{equation}
\end{proof}

\begin{proof}[Proof of \cref{theorem_no_linear_bounded} when   \(p>\ell+1\) and \(m=\ell+1\)]
Since the homotopy group \(\pi_\ell({\manifold N})\) is infinite, there exists a sequence \((v_j)_{j \in \Nset}\) in \(C^\infty (\overline\Bset^\ell, {\manifold N})\) such that each \(v_j\) is constant on \(\Sset^{\ell-1}\) and \(v_j\) is not homotopic with \(v_k\) if \(j\neq k\). Since \({\manifold N}\) is connected and \(v_j\) is constant on \(\Sset^{\ell-1}\), we may assume with no loss of generality that \(v_j=b\) on \(\Sset^{\ell-1}\), \(\forall\, j\). Consider now some  map \(w_j\in C^\infty_b(\overline\Bset^\ell, {\manifold N})\) such that \(w_j\sim v_j\) and
\begin{equation}
\label{c7}
{\mathcal{E}}^{1,p}(w_j)\le 2\,{\mathcal{E}}^{1,p}_{\mathrm{top}}(v_j)=2\,{\mathcal{E}}^{1,p}_{\mathrm{top}}(w_j).
\end{equation}
This is clearly possible, from the definition of \({\mathcal{E}}^{1,p}_{\mathrm{top}}\) and the density of \(C^\infty_b(\overline\Bset^\ell, {\manifold N})\) in \(W^{1,p}_b(\overline\Bset^\ell, {\manifold N})\).\footnote{Recall that \(p>\ell\) and thus \(W^{1,p}(\Bset^\ell, {\manifold N})\subset C(\overline\Bset^\ell, {\manifold N})\). Using this, the  density of \(C^\infty_b(\overline\Bset^\ell, {\manifold N})\) in \(W^{1,p}(\Bset^\ell, {\manifold N})\) is straightforward; see e.g. \cite{bethuelacta}*{Introduction}, and also the proof of
\cref{theorem_no_linear_bounded} when \(p=\ell+1\).
}
We claim that
\begin{equation}
\label{c5}
\lim_{j\to\infty}{\mathcal{E}}^{1-\sfrac 1p,p}(w_j)=\infty.
\end{equation}
Indeed, argue by contradiction and assume that \eqref{c5} does not hold. Using the Morrey type embedding \(W^{1-\sfrac 1p,p}(\Bset^\ell)\subset C^{0, 1-\sfrac {(\ell+1)}p}(\overline\Bset^\ell)\), we find that, up to a subsequence, the sequence \((w_j)_{j \in \Nset}\) converges uniformly on \(\overline \Bset^\ell\), and thus for large \(j\) and \(k\) we have \(w_j\sim w_k\), which is impossible.

We next modify \(w_j\)  by setting
\[u_j(x)\defeq\begin{cases}
w_j(2x),&\text{if }\abs{x}\le \sfrac{1}{2},\\
b,&\text{if } \abs{x}>\sfrac{1}{2},
\end{cases}
\]
and note that \(u_j\sim w_j\sim v_j\).
From the above, \(u_j\) satisfies \eqref{b1} with \(\rho =\sfrac{1}{2}\) and, in addition,
\begin{gather}
\label{d1}
{\mathcal{E}}^{1,p}(u_j)\le \Cl{d2}\,{\mathcal{E}}^{1,p}_{\mathrm{top}}(u_j)
\\
\shortintertext{and}
\label{d3}
\lim_{j\to\infty}{\mathcal{E}}^{1-\sfrac 1p,p}(u_j)=\infty.
\end{gather}
To summarize, if \(\pi_\ell({\manifold N})\) is infinite and \(p>\ell+1\), then there exists a sequence \((u_j)_{j \in \Nset}\) in \(C^\infty_b (\overline\Bset^{\ell}, {\manifold N})\) satisfying
 \eqref{d1}, \eqref{d3} and \eqref{b1} with \(\rho=\sfrac{1}{2}\).

 We next invoke the following  fractional  Gagliardo--Nirenberg type inequality
\begin{equation}
  \label{d5}
  \mathcal{E}^{1, 1 - \sfrac{1}{p}} (u) \le \Cl{d6} [\mathcal{E}^{1, p} (u)]^{1 - \sfrac{1}{p}} \norm{u}_{L^\infty}^{\sfrac 1p}
 \le \Cl{d7}  [\mathcal{E}^{1, p} (u)]^{1 - \sfrac{1}{p}},\ \forall\, u\in W^{1,p}(\Bset^\ell, {\manifold N})
\end{equation}
(see  e.g. \citelist{\cite{Runst_1986}*{Lemma 2.1}\cite{Brezis_Mironescu_2001}*{Corollary 3.2}\cite{Mazya_Shaposhnikova_2002}}).
Combining \eqref{c4}, \eqref{d1} and \eqref{d5}, we obtain the following \emph{superlinear lower bound}
\begin{equation}
\label{d8}
{\mathcal{E}}^{1,p}_{\mathrm{ext}}(u_j)\ge \Cl{d9} [{\mathcal{E}}^{1-\sfrac 1p,p}(u_j)]^{\sfrac p{(p-1)}},\ \forall\, j.
\end{equation}

We obtain the conclusion of the theorem from \eqref{d8} and \eqref{d3}. (Strictly speaking, the mapping \(u_j\) is only Lipschitz-continuous. However, using a standard approximation procedure, we obtain a sequence \((u_j)\subset C^\infty_b (\overline\Bset^{\ell}, {\manifold N})\) such that \eqref{d8} and \eqref{d3} hold.)
\end{proof}

\begin{proof}[Proof of \cref{theorem_no_linear_bounded} when \(p>\ell+1\) and \(m>\ell+1\)] The main idea is to proceed to a dimensional reduction. To illustrate this, consider the maps \(f_j(y, z)\defeq u_j(y)\), \(\forall\, (y,z)\in \Bset^\ell\times (0, 1)^{m-\ell-1}\) (with \(u_j\) as above). Via a Fubini type argument, it is easy to see that
\begin{equation}
\label{e3}
{\mathcal{E}}^{1,p}_{\mathrm{ext}}(f_j)\ge {\mathcal{E}}^{1,p}_{\mathrm{ext}}(u_j).
\end{equation}
On the other hand, by a direct calculation we have
\begin{equation}
\label{e4}
\Cl{e1} \, {\mathcal{E}}^{1-\sfrac 1p,p}(u_j)\le {\mathcal{E}}^{1-\sfrac 1p,p}(f_j)\le \Cl{e2} \, ({\mathcal{E}}^{1-\sfrac 1p,p}(u_j)+1).
\end{equation}
Combining \eqref{e3}--\eqref{e4} with the properties of \(u_j\), we find that \(f_j\) satisfies \eqref{e5}.

However, this \(f_j\) does not equal \(b\) on \(\Sset^{\ell-1}\). In order to obtain a map with this additional property,  we replace, in  the above construction, \((0,1)^{m-\ell-1}\) with a convenient sphere. The main ingredient is the existence of
some \(\Phi\in C^\infty (\overline\Bset^\ell\times \Sset^{m-\ell-1}, \Bset^{m-1})\)
such that \(\Phi\) is a diffeomorphism into its image \(V\). Taking this for granted, we argue as follows. Let \(u_j\) be as above, and set \(g_j(y,z)\defeq u_j (y)\), \(\forall\, y\in \overline\Bset^\ell\), \(\forall\, z\in \Sset^{m-\ell-1}\),  and
\begin{equation*}
h_j(x)\defeq\begin{cases}
g_j(\Phi^{-1}(x)),&\text{if }x\in V\\
b,&\text{if }x\in\overline\Bset^{m-1}\setminus V
\end{cases}.
\end{equation*}
Then \(h_j\in C^\infty_b (\overline
\Bset^{m-1}, {\manifold N})\). By adapting the arguments leading to \eqref{e3} and \eqref{e4}, we find that \((h_j)_{j \in \Nset}\) has the required properties.

It remains to prove the existence of \(\Phi\). Consider, for \(z\in \Sset^{m-\ell-1}\),  the following vectors in \(\Rset^{m-1}\):
\begin{equation*}
X_1\defeq (z/\abs{z},0),\,  X_2\defeq e_{m-\ell+1}\, \ldots,\,  X_{\ell}\defeq e_{m-1}.\footnoteref{f2}
\end{equation*}
\footnotetext[1]{\label{f2}Here, \((e_1,\ldots, e_{m-1})\) is the canonical base of \(\Rset^{m-1}\).}
Then, for sufficiently small \(\varepsilon >0\), the mapping
\begin{equation*}
\overline\Bset^{\ell}\times \Sset^{m-\ell-1} \ni (y,z)\mapsto  \Phi(y,z)\defeq\frac{1}{2}(z,0)+\varepsilon \sum_{k=1}^\ell y_k\, X_k\in \Rset^{m-1}
\end{equation*}
has the required properties.
\end{proof}

\begin{proof}[Proof of \cref{theorem_no_linear_bounded} when \(p=\ell+1\)]
As explained above, it suffices to consider the case \(m=\ell+1\).

\medbreak
If we examine the proof of \eqref{c4}, we see that the following lower bound is valid for any \(p\). If
\begin{equation}
\label{f3}
u\in W^{1-\sfrac 1p,p}(\Bset^{\ell}, {\manifold N})\cap C_b(\overline\Bset^{\ell}, {\manifold N})
\end{equation}
satisfies \eqref{b1}, then
then
\begin{multline}
\label{f4}
\int_{\Bset^{\ell+1}_+}\abs{\Deriv U}^p\ge \Cr{c2}\int_\rho^1 r^{\ell-p}\dif r\, {\mathcal{E}}^{1,p}_{\mathrm{top}}(u),\ \forall\, U\in W^{1,p}(\Bset^{\ell+1}_+, {\manifold N})\cap C(\overline{\Bset^{\ell+1}_+}, {\manifold N})
\\
\text{such that }\trace_{\Bset^{\ell}\times\{0\}}U\simeq u.
\end{multline}

The key observation is that, when \(p=\ell+1\) and \(u\) is, in addition, Lipschitz-continuous,  \eqref{f4} holds even if \(U\) is not supposed continuous, i.e.,
\begin{multline}
\label{f4a}
\int_{\Bset^{\ell+1}_+}\abs{\Deriv U}^{\ell+1}\ge \Cr{c2}{\mathcal{E}}^{1,\ell+1}_{\mathrm{top}}(u)\int_\rho^1 r^{-1}\dif r\, ,\ \forall\, U\in W^{1,\ell+1}(\Bset^{\ell+1}_+, {\manifold N})
\\
\text{such that }\trace_{\Bset^{\ell}\times\{0\}}U\simeq u.
\end{multline}
This is obtained by proving that (under these assumptions on \(p\) and \(u\)) for any map   \(U\in W^{1,p}(\Bset^{\ell+1}_+, {\manifold N})\) such that \(\displaystyle\trace_{\Bset^{\ell}\times\{0\}}U\simeq u\) there exists a sequence \((U_j)_{j \in \Nset}\) of mappings in \(C^\infty (\overline{\Bset^{\ell+1}_+}, {\manifold N})\) such that \(U_j\to U\) strongly in \(W^{1,p}\), \(U_j(\cdot, 0)\sim u\) for large \(u\) and \(U_j(x,0)=b\) if \(\rho\le \abs{x}\le 1\).

\medbreak
Here is a sketch of proof of this fact, well-known to experts and reminiscent from the theory of vanishing mean oscillation (VMO) maps with values into manifolds (see H. \familyname{Brezis} and L. \familyname{Nirenberg} \cite{Brezis_Nirenberg_1995}). First, we extend \(U\) to \(\Bset^{\ell+1}\setminus\Bset^{\ell+1}_+\) by setting \(U(x,t)=u(x)\) if \(t\le 0\). Next, we extend \(U\) by reflexion across \(\Sset^\ell\). We may thus assume that \(U\in W^{1,\ell+1}(B(0, \sfrac{3}{2}), {\manifold N})\). We next consider \(V_j(x,t)\defeq U((1+\sfrac 1j) x, t-\sfrac 1j)\). For large \(j\), \(V_j\) is defined in \(B(0, \sfrac{4}{3})\) and satisfies \(V_j(x,t)=u((1+\sfrac 1j) x)\) if \(\abs{x}\le \sfrac 54\) and \(\abs{t}\le \sfrac 1j\). In addition, we have \(V_j(x,t)=b\) if \( \rho/(1+\sfrac 1j)\le \abs{x}\le 1+\sfrac 1j\) and \(\abs{t}\le \sfrac 1j\). Consider now a standard mollifier \(\zeta\in C^\infty_c(\overline\Bset^{\ell+1}, \Rset)\) and let \(\Pi\) denote the nearest point projection on \({\manifold N}\). Then \(\Pi (V_j\ast\zeta_\varepsilon)\to V_j\) in \(W^{1,\ell+1}(\Bset^{\ell+1}_+, {\manifold N})\) as \(\varepsilon\to 0\) \cite{Brezis_Nirenberg_1995}\footnote{Here, we use the embedding \(W^{1,\ell+1}(\Rset^{\ell+1})\subset \mathrm{VMO}(\Rset^{\ell + 1})\)}. We easily find that, for a suitable sequence \(\varepsilon_j\to 0\), \(U_j\defeq \Pi (V_j\ast\zeta_{\varepsilon_j})\) has all the required properties.
\medbreak
We complete the case \(p=\ell+1\) and \(m=\ell+1\) as follows.
Since \(\pi_{\ell} (\manifold{N}) \not \simeq \{0\}\),  there exists
a map \(v \in C^\infty_b (\overline\Bset^\ell, \manifold{N})\) which such \(v=b\) that near \(\Sset^{\ell-1}\) and \(v\not\sim b\).
We claim that
\begin{equation}
\label{f4b}
{\mathcal{E}}^{1,p}_{\mathrm{top}} (v)>0.
\end{equation}
Indeed, argue by contradiction and assume that there exists a sequence of maps  \((v_j)_{j \in \Nset}\) in \(W^{1, \ell+1}_b(\Bset^\ell, {\manifold N})\) such that \(v_j\sim v\), \(\forall\, j\) and \({\mathcal{E}}^{1,p}(v_j)\to 0\). By the Morrey embedding \(W^{1,\ell+1}(\Bset^\ell)\subset C^{0, 1-\sfrac \ell{(\ell+1)}}(\overline\Bset^\ell)\) and the fact that \(v_j=b\) on \(\Sset^{\ell-1}\), we find that  \(v_j\to b\) uniformly, and thus, for large \(j\), \(v\sim v_j\sim b\), a contradiction.

\medbreak
We define for \(\rho \in (0, 1)\) the  map   \(u_\rho\in C^\infty(\Rset^\ell,  \manifold{N})\) by
\[
u_\rho (x)
\defeq
\begin{cases}
  v (x/\rho), & \text{if \(\abs{x} \le \rho\)}\\
  b, & \text{otherwise}
\end{cases},
\]
whose restriction to \(\Bset^{\ell}\), still denoted \(u_\rho\), satisfies
\begin{equation}
  \label{g3}
  \begin{aligned}
\mathcal{E}^{1 - \sfrac{1}{(\ell+1)}, \ell+1} (u_\rho)&=\iint\limits_{\Bset^{\ell} \times \Bset^{\ell}}
\frac{\abs{u_\rho (y) - u_\rho (x)}^{\ell+1}}{\abs{y - x}^{2\ell}}\dif y \dif x
\\
&\le  \iint\limits_{\Rset^{\ell} \times \Rset^{\ell}}
\frac{\abs{u_\rho (y) - u_\rho (x)}^{\ell+1}}{\abs{y - x}^{2\ell}}\dif y \dif x\\
&= \iint\limits_{\Rset^{\ell} \times \Rset^{\ell}}
\frac{\abs{u (y) - u (x)}^{\ell+1}}{\abs{y - x}^{2\ell}}\dif y \dif x=
\C <\infty.
\end{aligned}
\end{equation}
On the other hand, we clearly have
\begin{equation}
\label{f6}
\lim_{\rho\to 0+}{\mathcal{E}}^{1-\sfrac 1{(\ell+1)}, \ell+1}(u_\rho)=\C.
\end{equation}
Since \(u_\rho\sim v\), we obtain, from \eqref{f4a}--\eqref{g3}, that
\begin{equation}
\label{f7}
{\mathcal{E}}^{1, \ell+1}_{\mathrm{ext}}(u_\rho)\ge \Cl{f8}\, \ln \frac 1\rho\, {\mathcal{E}}^{1-\sfrac 1{(\ell+1)}, \ell+1}(u_\rho).
\end{equation}

We complete the proof of the theorem in this case via \eqref{f6} and \eqref{f7}.
\end{proof}

We now deduce \cref{theorem_analytical_obstruction} from \cref{theorem_no_linear_bounded}.

\begin{proof}[Proof of \cref{theorem_analytical_obstruction}]
  By \cref{theorem_no_linear_bounded} and an extension argument for fractional Sobolev spaces, there exists a sequence of mappings \((u_j)_{j \in \Nset}\) in \(C^\infty (\Rset^{m - 1}, \manifold{N})\) such that for every \(j \in \Nset\), \(u_j = b\) on \(\Rset^{m - 1} \setminus \Bset^{m - 1}\),
  \begin{equation}
  \label{h1}
  \mathcal{E}^{1 - \sfrac{1}{p}, p} (u_j) \ge
  \Cl{g1}\text{ and }
  \mathcal{E}^{1, p}_{\mathrm{ext}} ({u_j}\restr{ \Bset^{m - 1}}) \ge 2^j\, \mathcal{E}^{1 - \sfrac{1}{p}, p} (u_j)
  \end{equation}
  for some some constant \(\Cr{g1} > 0\).
  We fix the radii \(r_j > 0\) by the condition
  \begin{equation}
    \label{h2}
  r_j^{m - p}\, \mathcal{E}^{1, p}_{\mathrm{ext}} (u_j) = 1.
  \end{equation}
  Since, by assumption \(p < m\), we have \(r_j \le \C 2^{-j/(m - p)}\),
  so that, in particular,
  \[
  \sum_{j \in \Nset} r_j^{m - 1} <\infty.
  \]
Therefore, we may find
 some  \(0<\lambda <\infty\) and a sequence of points \((a_j)_{j \in \Nset}\) in \(\Bset^{m - 1}\) converging to a point of \(a_* \in \Sset^{m - 1}\)
  such that the balls \(B (a_j, \lambda r_j)\) are mutually disjoint and contained in \(\Bset^{m - 1}\).
We  then define the map \(u \in C^\infty (\Bset^{m - 1}, \manifold{N})\) by setting, for \(x \in \Bset^{m - 1}\),
  \begin{equation*}
   u (x)
   \defeq
   \begin{cases}
     u_j ( \sfrac1 {(\lambda r_j)} \, (x-a_j)), & \text{if \(x \in B (a_j, \lambda r_j)\) for some \(j \)}\\
     b, & \text{otherwise}
   \end{cases}.
  \end{equation*}

  By the superadditivity of the  extension energy and  \eqref{h2}, we have
\begin{equation}
\label{j1}
\mathcal{E}^{1, p}_{\mathrm{ext}} (u)
\ge \sum_{j \in \Nset} (\lambda r_j)^{m - p}\,  \mathcal{E}^{1, p}_{\mathrm{ext}} (u_j)
\ge \sum_{j \in \Nset} \lambda^{m - p} = \infty.
\end{equation}
On the other hand, we have,  by the almost subadditivity property for Sobolev mappings having disjoint supports \cite{Monteil_Van_Schaftingen_2019}*{Lemma 2.3},  \eqref{h1} and \eqref{h2}:
\begin{equation}
\label{j2}
\begin{aligned}
\mathcal{E}^{1 - \sfrac{1}{p}, p} (u)&\le \Cl{i1} \sum_{j \in \Nset} (\lambda r_j)^{m-p+1}{\mathcal{E}^{1 - \sfrac{1}{p}, p} (u_j)\le \Cr{i1}\, \lambda^{m-p+1}\sum_{j \in \Nset} 2^{-j}\, r_j}\\
&\le \Cr{i1}\, \lambda^{m-p+1}\sum_{j \in \Nset} 2^{-j}<\infty.
\end{aligned}
\end{equation}

By \eqref{j1} and \eqref{j2},  \(u\) is a \(W^{1-\sfrac 1p,p}(\Bset^\ell, {\manifold N})\) map with no \(W^{1,p}(\Bset^\ell, {\manifold N})\) extension.
\end{proof}

\section{Construction of extension}
\label{section_extension}
We explain how \cref{theorem_trace_sufficient_medium_p} and \cref{theorem_extension_estimates}
follow from existing results on extension for simply-connected manifolds through a lifting argument; this important observation is due to F. \familyname{Bethuel} \cite{Bethuel_2014}.

\begin{proof}[Proof of  the new cases in \cref{theorem_trace_sufficient_medium_p} and \cref{theorem_extension_estimates}]
  \resetconstant%
  Let \(\pi : {\widetilde{\manifold{N}}} \to \manifold{N}\) be the universal covering of the manifold \(\manifold{N}\).
  Since the fundamental group \(\pi_1 (\manifold{N})\) is finite, the universal covering space \(\widetilde{\manifold{N}}\) is compact;
  in view of the fractional lifting theorem for compact covering spaces \cite{Mironescu_VanSchaftingen},
  for every \(u \in W^{1-\sfrac{1}{p}, p} (\partial \manifold{M}, {\manifold{N}})\)
  there exists a map \(\widetilde{u} \in W^{1-\sfrac{1}{p}, p} (\partial \manifold{M}, {\widetilde{\manifold{N}}})\) such that \(\pi \compose \widetilde{u} = u\) in \(\manifold{M}\)
  and
  \[
  \mathcal{E}^{1 - \sfrac{1}{p}, p} (\widetilde{u})
  \le
  \Cl{cst_proof_ext_1}\,
  \mathcal{E}^{1 - \sfrac{1}{p}, p} (u),
  \]
  for some constant \(\Cr{cst_proof_ext_1}\) independent on the mapping \(u\).

  Since \({\widetilde{\manifold{N}}}\) is the universal covering of \(\manifold{N}\), it is simply-connected (that is, \(\pi_{1} ({\widetilde{\manifold{N}}}) \simeq\{0\}\)) and it has the same higher-order homotopy groups as \(\manifold{N}\): for every \(j \in \{2, \dotsc, \floor{p - 1}\}\), we have \(\pi_j ({\widetilde{\manifold{N}}}) \simeq \pi_j (\manifold{N}) \simeq \{0\}\) (see for example \cite{Hatcher_2002}*{Proposition 4.1}).
  By \cref{theorem_trace_sufficient_small_p} and \cref{theorem_extension_estimates} \emph{1} (applied to the old case where  \(\pi_1({\manifold N})\simeq\{0\}\)),
  there exists a mapping \(\widetilde{U} \in W^{1, p} (\Bset^{m - 1} \times (0, 1), {\widetilde{\manifold{N}}})\) with trace \(\widetilde{u}\) and such that
  \[
  \mathcal{E}^{1, p} (\widetilde{U})
  \le
  \Cl{cst_proof_ext_2}\,
  \mathcal{E}^{1 - \sfrac{1}{p}, p}(\widetilde{u})
  .
  \]

  We conclude by defining \(U \defeq \pi \compose \widetilde{U}\). Since the covering map \(\pi\) is a local isometry, we have \(u = \pi \compose \widetilde{u}\) on \(\Bset^{m - 1}\) and \(u = \trace_{\Bset^{m - 1} \times \{0\}}  U\), \(U \in W^{1, p} (\Bset^{m - 1} \times (0, 1), \manifold{N})\) and
  \[
  \mathcal{E}^{1, p} (U)
  =
  \mathcal{E}^{1, p} (\widetilde{U})
  \le
  \Cr{cst_proof_ext_2}\,
  \mathcal{E}^{1 - \sfrac{1}{p}, p} (\widetilde{u})
  \le
  \Cr{cst_proof_ext_2}\,
  \Cr{cst_proof_ext_1}\,
  \mathcal{E}^{1 - \sfrac{1}{p}, p} (u)
  .
  \qedhere
  \]
\end{proof}

\section{Manifolds on which the problem is open}

The next proposition shows the existence of compact manifolds with finitely many prescribed homotopy groups.
This is a straightforward and probably well-known variant of the product of Eilenberg--McShane spaces giving
CW complexes with an arbitrary sequence of homotopy groups \cite{Hatcher_2002}*{\S 4.2}.
The interest of the next proposition is that the resulting space is a compact finite-dimensional manifold.

\begin{proposition}
  \label{proposition_manifold_EMcL}
If \(\ell \in \Nset\) and \(G_1, \dotsc, G_\ell\) are finitely generated groups, and if \(G_2, \dotsc, G_n\) are abelian, then
there exists a \(2 (\ell + 1)\)--dimensional compact manifold \(\manifold{N}\)
such that for every \(j \in \{1, \dotsc, \ell\}\), \(\pi_j (\manifold{N}) = G_j\).
\end{proposition}
\begin{proof}
We define
\[
 X \defeq K (G_1, 1) \times \dotsb \times K (G_\ell, \ell),
\]
where the Eilenberg-McLane space \(K (G_j, j)\) is a CW-complex of finite type whose only non-trivial homotopy group is \(\pi_j (K (G_j, j)) = G_j\) \cite{Hatcher_2002}*{\S 4.2}.
We then have \(\pi_j (X) = G_j\) for every \(j \in \{1, \dotsc, \ell\}\).
Let \(X_{\ell + 1}\) be the component of \(X\) consisting of cells of dimensions at most \(\ell + 1\).
It follows then that \(\pi_j (X_{\ell+1}) = G_j\) for every \(j \in \{1, \dotsc, \ell\}\).
Since \(X\) is of finite type, \(X_{\ell + 1}\) is a finite CW-complex, that can be realized as a simplicial complex \(K\) of dimension \(\ell + 1\).
We embed \(K\) in the Euclidean space  \(\Rset^\nu\) with \(\nu = 2 \ell + 3\) and we let \(\manifold{N} \defeq \partial \mathcal{U}\), where   \(\mathcal{U}\) is a  smooth   neighborhood of \(K\) that retracts on \(K\) and such that \(\mathcal{U} \setminus K\) retracts on \(\manifold{N}\).
Since \(K\) is of dimension \(\ell + 1\), it follows that for every \(j \in \{1, \dotsc, \ell\}\), any continuous  map \(f : \Bset^{j+ 1} \to \manifold{U}\) such that \(f\restr{{\Sset^j}}\) takes its values in \(\manifold{N}\)  is homotopic to a map with values in \(\mathcal{U} \setminus K\),  and thus
\(\pi_j (\manifold{N}) = \pi_j (K) = \pi_j (X_{\ell + 1}) = G_j\).
\end{proof}


\begin{bibdiv}
\begin{biblist}

\bib{Adams_Fournier_2003}{book}{
    author={Adams, Robert A.},
    author={Fournier, John J. F.},
    title={Sobolev spaces},
    series={Pure and Applied Mathematics},
    volume={140},
    edition={2},
    publisher={Elsevier/Academic Press},
    address={Amsterdam},
    date={2003},
    pages={xiv+305},
    isbn={0-12-044143-8},
  }

  \bib{Bethuel_1990}{article}{
    author={Bethuel, Fabrice},
    title={A characterization of maps in \(H^1(\Bset^3,\Sset^2)\) which can be
      approximated by smooth maps},
    journal={Ann. Inst. H. Poincar\'{e} Anal. Non Lin\'{e}aire},
    volume={7},
    date={1990},
    number={4},
    pages={269--286},
    issn={0294-1449},
    doi={10.1016/S0294-1449(16)30292-X},
  }

\bib{bethuelacta}{article}{
author={Bethuel, Fabrice},
title={The approximation problem for {S}obolev maps between two manifolds},
journal={ Acta Math.},
volume={167},
date={1991},
number={3-4},
pages={153-206}
}

  \bib{Bethuel_2014}{article}{
    author={Bethuel, Fabrice},
    title={A new obstruction to the extension problem for Sobolev maps
      between manifolds},
    journal={J. Fixed Point Theory Appl.},
    volume={15},
    date={2014},
    number={1},
    pages={155--183},
    issn={1661-7738},
    doi={10.1007/s11784-014-0185-0},
  }

  \bib{Bethuel_Chiron_2007}{article}{
    author={Bethuel, Fabrice},
    author={Chiron, David},
    title={Some questions related to the lifting problem in Sobolev spaces},
    conference={
      title={Perspectives in nonlinear partial differential equations},
    },
    book={
      series={Contemp. Math.},
      volume={446},
      publisher={Amer. Math. Soc., Providence, R.I.},
    },
    date={2007},
    pages={125--152},
    %    review={\MR{2373727}},
    doi={10.1090/conm/446/08628},
  }

\bib{Bethuel_Coron_Demengel_Helein_1991}{article}{
  author={Bethuel, Fabrice},
  author={Coron, Jean-Michel},
  author={Demengel, Fran\c coise},
  author={H\'{e}lein, Fr\'ed\'eric},
  title={A cohomological criterion for density of smooth maps in Sobolev
    spaces between two manifolds},
  conference={
    title={Nematics},
    address={Orsay},
    date={1990},
  },
  book={
    series={NATO Adv. Sci. Inst. Ser. C Math. Phys. Sci.},
    volume={332},
    publisher={Kluwer Acad. Publ., Dordrecht},
  },
  date={1991},
  pages={15--23},
}


\bib{Bethuel_Demengel_1995}{article}{
   author={Bethuel, Fabrice},
   author={Demengel, Fran\c coise},
   title={Extensions for Sobolev mappings between manifolds},
   journal={Calc. Var. Partial Differential Equations},
   volume={3},
   date={1995},
   number={4},
   pages={475--491},
   issn={0944-2669},
   doi={10.1007/BF01187897},
}
		
\bib{Bourgain_Brezis_Mironescu_2000}{article}{
   author={Bourgain, Jean},
   author={Brezis, Ha\"\i m},
   author={Mironescu, Petru},
   title={Lifting in Sobolev spaces},
   journal={J. Anal. Math.},
   volume={80},
   date={2000},
   pages={37--86},
   issn={0021-7670},
   doi={10.1007/BF02791533},
}

\bib{Bourgain_Brezis_Mironescu_2004}{article}{
   author={Bourgain, Jean},
   author={Brezis, Ha\"\i m},
   author={Mironescu, Petru},
   title={\(H^{\sfrac{1}{2}}\) maps with values into the circle: minimal connections, lifting, and the Ginzburg-Landau equation},
   journal={Publ. Math. Inst. Hautes \'Etudes Sci.},
   number={99},
   date={2004},
   pages={1--115},
   issn={0073-8301},
   doi={10.1007/s10240-004-0019-5},
}

\bib{Bousquet_Ponce_Van_Schaftingen_2014}{article}{
   author={Bousquet, Pierre},
   author={Ponce, Augusto C.},
   author={Van Schaftingen, Jean},
   title={Strong approximation of fractional Sobolev maps},
   journal={J. Fixed Point Theory Appl.},
   volume={15},
   date={2014},
   number={1},
   pages={133--153},
   issn={1661-7738},
   doi={10.1007/s11784-014-0172-5},
}
		
\bib{Brezis_Mironescu_2001}{article}{
  author={Brezis, Ha\"{i}m},
  author={Mironescu, Petru},
  title={Gagliardo--Nirenberg, composition and products in fractional
    Sobolev spaces},
  %    note={Dedicated to the memory of Tosio Kato},
  journal={J. Evol. Equ.},
  volume={1},
  date={2001},
  number={4},
  pages={387--404},
  issn={1424-3199},
  %    review={\MR{1877265}},
  doi={10.1007/PL00001378},
}

\bib{Brezis_Mironescu_2015}{article}{
   author={Brezis, Ha\"\i m},
   author={Mironescu, Petru},
   title={Density in \(W^{s,p}(\Omega;\manifold{N})\)},
   journal={J. Funct. Anal.},
   volume={269},
   date={2015},
   number={7},
   pages={2045--2109},
   issn={0022-1236},
   doi={10.1016/j.jfa.2015.04.005},
}

 \bib{bmbook}{book}{
      author={Brezis, Ha\"\i m},
      author={Mironescu, Petru},
      title={Sobolev maps to the circle},
      note={In preparation},
    }


\bib{Brezis_Nirenberg_1995}{article}{
  author={Brezis, Ha\"\i m},
  author={Nirenberg, Louis},
  title={Degree theory and BMO},
  part={I},
  subtitle={Compact manifolds without boundaries},
  journal={Selecta Math. (N.S.)},
  volume={1},
  date={1995},
  number={2},
  pages={197--263},
  issn={1022-1824},
  doi={10.1007/BF01671566},
}

\bib{diBenedetto_2016}{book}{
  author={DiBenedetto, Emmanuele},
  title={Real analysis},
  series={Birkh\"{a}user Advanced Texts: Basler Lehrb\"{u}cher},
  edition={2},
  publisher={Birk\-h\"{a}user/Springer},
  address={New York},
  date={2016},
  pages={xxxii+596},
  isbn={978-1-4939-4003-5},
  isbn={978-1-4939-4005-9},
  %    review={\MR{3560412}},
  doi={10.1007/978-1-4939-4005-9},
}

\bib{Gagliardo_1957}{article}{
  author={Gagliardo, Emilio},
  title={Caratterizzazioni delle tracce sulla frontiera relative ad alcune
    classi di funzioni in \(n\) variabili},
  journal={Rend. Sem. Mat. Univ. Padova},
  volume={27},
  date={1957},
  pages={284--305},
  issn={0041-8994},
}

\bib{Hardt_Lin_1987}{article}{
   author={Hardt, Robert},
   author={Lin, Fang-Hua},
   title={Mappings minimizing the \(L^p\) norm of the gradient},
   journal={Comm. Pure Appl. Math.},
   volume={40},
   date={1987},
   number={5},
   pages={555--588},
   issn={0010-3640},
   doi={10.1002/cpa.3160400503},
 }

 \bib{Hatcher_2002}{book}{
   author={Hatcher, Allen},
   title={Algebraic topology},
   publisher={Cambridge University Press},
   address={Cambridge},
   date={2002},
   pages={xii+544},
   isbn={0-521-79160-X},
   isbn={0-521-79540-0},
 }

 \bib{Isobe_2003}{article}{
   author={Isobe, Takeshi},
   title={Obstructions to the extension problem of Sobolev mappings},
   journal={Topol. Methods Nonlinear Anal.},
   volume={21},
   date={2003},
   number={2},
   pages={345--368},
   issn={1230-3429},
   doi={10.12775/TMNA.2003.021},
 }

 \bib{Lions_Peetre_1964}{article}{
   author={Lions, Jacques-Louis},
   author={Peetre, Jaak},
   title={Sur une classe d'espaces d'interpolation},
   journal={Inst. Hautes \'{E}tudes Sci. Publ. Math.},
   number={19},
   date={1964},
   pages={5--68},
   issn={0073-8301},
 }

 \bib{Mazya_2011}{book}{
   author={Maz'ya, Vladimir},
   title={Sobolev spaces with applications to elliptic partial differential equations},
   series={Grundlehren der Mathematischen Wissenschaften},
   volume={342},
   edition={2}, 
   publisher={Springer},
   address={Heidelberg},
   date={2011},
   pages={xxviii+866},
   isbn={978-3-642-15563-5},
   doi={10.1007/978-3-642-15564-2},
 }

 \bib{Mazya_Shaposhnikova_2002}{article}{
   author={Maz\cprime ya, Vladimir},
   author={Shaposhnikova, Tatyana},
   title={On the Brezis and Mironescu conjecture concerning a
     Gagliardo-Nirenberg inequality for fractional Sobolev norms},
   journal={J. Math. Pures Appl. (9)},
   volume={81},
   date={2002},
   number={9},
   pages={877--884},
   issn={0021-7824},
   doi={10.1016/S0021-7824(02)01262-X},
 }

 \bib{Mironescu_Russ_2015}{article}{
   author={Mironescu, Petru},
   author={Russ, Emmanuel},
   title={Traces of weighted Sobolev spaces. Old and new},
   journal={Nonlinear Anal.},
   volume={119},
   date={2015},
   pages={354--381},
   issn={0362-546X},
   doi={10.1016/j.na.2014.10.027},
 }

 \bib{Mironescu_VanSchaftingen}{article}{
 author={Mironescu, Petru},
 author={Van Schaftingen, Jean},
 title={Lifting of fractional Sobolev maps to compact covering spaces},
 eprint={arXiv:1907.01373},
 note={submitted for publication},
}

\bib{Monteil_Van_Schaftingen_2019}{article}{
     author={Monteil, Antonin},
     author={Van Schaftingen, Jean},
     title={Uniform boundedness principles for Sobolev maps into manifolds},
     journal={Ann. Inst. H. Poincar\'{e} Anal. Non Lin\'{e}aire},
     volume={36},
     date={2019},
     number={2},
     pages={417--449},
     issn={0294-1449},
     doi={10.1016/j.anihpc.2018.06.002},
}

\bib{Nash_1956}{article}{
  author={Nash, John},
  title={The imbedding problem for Riemannian manifolds},
  journal={Ann. of Math. (2)},
  volume={63},
  date={1956},
  pages={20--63},
  issn={0003-486X},
  doi={10.2307/1969989},
}

\bib{Peetre_1979}{article}{
  author={Peetre, Jaak},
  title={A counterexample connected with Gagliardo's trace theorem},
  journal={Comment. Math. Special Issue},
  volume={2},
  date={1979},
  pages={277--282},
}


\bib{Riviere_2000}{article}{
  author={Rivi\`ere, Tristan},
  title={Dense subsets of \(H^{\sfrac{1}{2}}(\Sset^2,\Sset^1)\)},
  journal={Ann. Global Anal. Geom.},
  volume={18},
  date={2000},
  number={5},
  pages={517--528},
  issn={0232-704X},
  doi={10.1023/A:1006655723537},
}

\bib{Runst_1986}{article}{
  author={Runst, Thomas},
  title={Mapping properties of nonlinear operators in spaces of
    Triebel-Lizorkin and Besov type},
  journal={Anal. Math.},
  volume={12},
  date={1986},
  number={4},
  pages={313--346},
  issn={0133-3852},
  doi={10.1007/BF01909369},
}

\bib{Schoen_Uhlenbeck_1982}{article}{
  author={Schoen, Richard},
  author={Uhlenbeck, Karen},
  title={A regularity theory for harmonic maps},
  journal={J. Differential Geom.},
  volume={17},
  date={1982},
  number={2},
  pages={307--335},
  issn={0022-040X},
}



\bib{Schoen_Uhlenbeck_1983}{article}{
  author={Schoen, Richard},
  author={Uhlenbeck, Karen},
  title={Boundary regularity and the Dirichlet problem for harmonic maps},
  journal={J. Differential Geom.},
  volume={18},
  date={1983},
  number={2},
  pages={253--268},
  issn={0022-040X},
}

\bib{Uspenskii_1961}{article}{
  author={Uspenski\u{\i}, S. V.},
  title={Imbedding theorems for classes with weights},
  language={Russian},
  journal={Trudy Mat. Inst. Steklov.},
  volume={60},
  date={1961},
  pages={282--303},
  issn={0371-9685},
  translation={
    journal={Am. Math. Soc. Transl.},
    volume={87},
    pages={121--145},
    date={1970},
  },
}



\bib{White_1988}{article}{
  author={White, Brian},
  title={Homotopy classes in Sobolev spaces and the existence of energy minimizing maps},
  journal={Acta Math.},
  volume={160},
  date={1988},
  number={1-2},
  pages={1--17},
  issn={0001-5962},
  doi={10.1007/BF02392271},
}


\end{biblist}

\end{bibdiv}
\end{document}